\documentclass{amsart}
\usepackage{amsfonts}
\usepackage{color}

\usepackage{graphicx}
\usepackage{enumerate}
 
\usepackage{hyperref}

\newtheorem{thm}{Theorem}[section]
\newtheorem{cor}[thm]{Corollary}
\newtheorem{lema}[thm]{Lemma}
\newtheorem{prop}[thm]{Proposition}
\theoremstyle{definition}
\newtheorem{defn}[thm]{Definition}
\theoremstyle{remark}
\newtheorem{rem}[thm]{Remark}
\theoremstyle{example}

\def\supp{\mathop{\text{\normalfont supp}}}

\numberwithin{equation}{section}
\newcommand{\R}{\mathbb R}

\newcommand{\N}{\mathbb N}

\def\C{\mathbf {C}}

\newcommand{\intr}{\int_{\R^n}}

\newcommand{\ve}{\varepsilon}
\newcommand{\lam}{\lambda}

\newcommand{\LL}{\mathcal{L}}

\begin{document}

\title{A P\'{o}lya-Szeg\"{o} principle for general fractional Orlicz-Sobolev spaces}
	
\author[P. De N\'apoli]{Pablo de N\'apoli}
\author[J. Fern\'andez Bonder]{Juli\'an Fern\'andez Bonder}
\author[A. Salort]{Ariel Salort}

\address{Departamento de Matem\'atica, FCEyN - Universidad de Buenos Aires and
\hfill\break \indent IMAS - CONICET
\hfill\break \indent Ciudad Universitaria, Pabell\'on I (1428) Av. Cantilo s/n. \hfill\break \indent Buenos Aires, Argentina.}

\email[P. De N\'apoli]{pdenapo@dm.uba.ar}
\urladdr{http://mate.dm.uba.ar/~pdenapo}

\email[J. Fern\'andez Bonder]{jfbonder@dm.uba.ar}
\urladdr{http://mate.dm.uba.ar/~jfbonder}

\email[A. Salort]{asalort@dm.uba.ar}
\urladdr{http://mate.dm.uba.ar/~asalort}



\begin{abstract}
In this article we prove modular and norm P\'{o}lya-Szeg\"{o} inequalities in general fractional Orlicz-Sobolev spaces by using the polarization technique. We introduce a general framework which includes the different definitions of theses spaces in the literature, and we establish  some of its basic properties such as the density of smooth functions.

\noindent As a corollary we prove a Rayleigh-Faber-Krahn type inequality for Dirichlet eigenvalues under nonlocal nonstandard growth operators.
\end{abstract}

\maketitle

\section{Introduction}
Symmetrization procedures have became a fundamental  tool in the history of isoperimetric problems, which go back to the works of J. Steiner \cite{steiner} (1838) and 	H. Schwarz \cite{schw} (1884). Given a  $u:\R^n\to \R^+\cup\{+\infty\}$, its \emph{symmetric rearrangement} or \emph{Schwarz symmetrization} is the function $u^*:\R^n\to \R^+\cup\{+\infty\}$ defined as the unique one such that for every $\lambda\geq 0$ there exists $R\geq 0$ such that 
\begin{align*}
&B(0,R)=\{x\in \R^n \colon u^*(x)>\lambda\},\\
\LL^n(\{x\in \R^n \colon &u^*(x)>\lambda\})=\LL^n(\{x\in \R^n \colon u(x)>\lambda\}).
\end{align*}
Therefore, the function $u^*$ is radial and radially decreasing and whose sub-level sets have the same measure as those of $u$. From this it is easily deduced that  if $u\in L^p(\R^n)$, then $u^*\in L^p(\R^n)$ and both functions have the same $L^p$ norm. The case of Sobolev functions is more subtle, and the celebrated inequality named after G. P\'{o}lya and G. Szeg\"{o} \cite{PoSz} (1951)  states that if $u\in W^{1,p}(\R^n)$ is nonnegative, then $u^*\in W^{1,p}(\R^n)$ and it is satisfied that
$$
\int_{\R^n} |\nabla u^*|^p \leq \int_{\R^n} |\nabla u|^p.
$$
This inequality is crucial in the proof of the Rayleigh-Faber-Krahn inequality, which asserts that balls minimize the first eigenvalue of the Dirichlet $p-$Laplacian among sets with given volume, that is, 
$$
\lambda_1(B) \leq \lambda_1(\Omega), \qquad \lam_1(\Omega):=\min_{u\in W^{1,p}_0(\Omega)} \left\{\int_\Omega |\nabla u|^p \,dx : \|u\|_{L^p(\Omega )} =1 \right\}
$$
where $B$ is a ball having the same measure as $\Omega$. We refer the reader to the survey \cite{Talenti} for more information on the  symmetric rearrangement.

\medskip

Other applications of the P\'olya-Szeg\"o principle arise in establishing sharp geometric and functional inequalities, such as the sharp Sobolev inequality, Hardy-Littlewood-Sobolev inequality, Moser-Trudinger inequality, etc. See for instance \cite{FS, Lieb, Milman, Talenti}, etc.

\medskip

The P\'{o}lya-Szeg\"{o} principle was extended to  fractional Sobolev spaces  in \cite{AL}, where the authors prove that $[u^*]_{W^{s,p}(\R^n)} \le [u]_{W^{s,p}(\R^n)}$ for any $u\in W^{s,p}(\R^n)$ with $p>1$ and $s\in (0,1)$. Thus, in \cite{BLP} the Rayleigh-Faber-Krahn inequality was generalized for $p-$fractional eigenvalues.

We remark that in the fractional case, this approach is not always very successful for obtaining sharp geometric and functional inequalities since the reduction to a one dimensional problem does not always give any evident advantage. As far as we are aware the only paper that treats a sharp fractional Sobolev inequality is \cite{Lieb} where, as a corollary of a more general Hardy-Littlewood-Sobolev inequality the case of the embedding of $W^{s,2}(\R^n)$ into $L^{2^*_s}(\R^n)$ is fully analyzed. 

Yet another interesting paper is \cite{FS} where, using this fractional version of the P\'olya-Szeg\"o principle, 
the authors deduce a sharp Sobolev inequality with Lorentz norms from the sharp fractional Hardy inequality.

\medskip

It is worth mentioning that recently some rearrangement-free methods have been developed to deal with problems where a P\'olya-Szeg\"o principle is not available. This is the case, for instance, of functional inequalities in the Heisenberg group. See \cite{Lam1, Lam2, Lam3}. It is an interesting open problem wether these methods can be used to obtain sharp geometric inequalities in the context of Orlicz-Sobolev spaces. We leave these questions to further investigations.

\medskip

The main goal of this paper is to extend that principle to the more general setting of fractional Orlicz-Sobolev spaces, thus allowing growth laws different than powers.  However, several definitions of fractional Orlicz-Sobolev have been  proposed in the literature in  \cite{ABS} \cite{FBS} and  \cite{KK}.   It turns out that essentially the 
same proof of the P\'{o}lya-Szeg\"{o} principle works for all of them. For this reason, we propose here 
a more general framework that encompasses these different definitions (see section \ref{subsection-examples} below for details).

\medskip

To be more precise, consider the class of Young functions $G:\R \to \R$, such that $g=G'$, which, basically consists in  even, convex and increasing functions. In addition, we assume $G$ to satisfy the growth condition
\begin{equation} 
\tag{$P_0$} \label{cond.intro}
1<p^-\leq \frac{tg(t)}{G(t)} \leq p^+ \quad \forall t>0,
\end{equation}
for some constants $p^\pm$. Given two   function $M,N:\R_+\to\R_+$  such that
\begin{equation} \tag{$P_1$}  \label{P1}
\text{$M$ and $N$ are nondecreasing and $M(r), N(r)>0$ for $r>0$,}
\end{equation}
\begin{equation} \tag{$
P_2$}  \label{P2}
 M(r) \geq \min\{1, r\} \text{ and } N \text{ is continuous},
\end{equation}
\begin{equation} \tag{$P_3$}   \label{P3}
\int_0^1 \frac{r^{n-1+p^-}}{N(r) M(r)^{p^-}} \,dr< \infty, \qquad \int_1^\infty \frac{r^{n-1}}{N(r)M(r)^{p^-} }\,dr <\infty,
\end{equation}
and any Young function $G$ satisfying \eqref{cond.intro}, we consider the general fractional Orlicz-Sobolev spaces defined  as
$$
W^{M,N,G}(\R^n):=\left\{ u\in L^G(\R^n) \text{ such that } \Phi_{M,N,G}(u)<\infty \right\}
$$
where the usual Orlicz space $L^G$ is defined as
$$
L^G(\R^n) :=\left\{ u\colon \R^n \to \R \text{ measurable, such that }  \Phi_{G}(u) < \infty \right\}
$$
and the modulars $\Phi_G$ and $\Phi_{M,N,G}$ are determined as
\begin{align*}
\Phi_{G}(u)&:=\int_{\R^n} G(u(x))\,dx,\\
\Phi_{M,N,G}(u)&:= \int_{\R^{2n}} G\left( \frac{u(x)-u(y)}{M(|x-y|)} \right) \frac{ dx\,dy}{N(|x-y|)}.
\end{align*}

\medskip

In this  context we prove a P\'olya-Szeg\"{o} principle for modulars, namely, 

\begin{thm} 
Consider an Young function $G$ satisfying \eqref{cond.intro}  and two increasing function $M$ and $N$ satisfying \eqref{P1}--\eqref{P3}. If $u\in W^{M,N,G}(\R^n)$, then $u^*\in W^{M,N,G}(\R^n)$ and
$$
\Phi_{M,N,G}(u^*) \leq \Phi_{M,N,G}(u).
$$ 
\end{thm}
Moreover, this inequality  also holds for the Luxemburg's norm in  $W^{M,N,G}(\R^n)$.

\medskip

Our proof relies on the symmetrization  via  polarization approach introduced in \cite{BrSo}, and on the construction in \cite{VS}. This technique requires the density of smooth functions with compact support in our space. We also give a detailed proof of that property in Section \ref{sec.density}.

\medskip

As a direct application of our main result we prove a Faber-Krahn type inequality for \emph{Dirichlet $G-$eigenvalues} and \emph{Poincar\'e's constants}.

The space $W^{M,N,G}(\R^n)$ is the natural one to define the the general fractional $g-$Laplacian operator $(-\Delta_g)^{M,N}$ as the gradient of the functional $\Phi_{M,N,G}$, which is well-defined between $W^{M,N,G}(\R^n)$ and its dual space. See \eqref{fracg} for an explicit formula.  

In order to define eigenvalues of this operator we will need to assume an additional growth condition on $M$ and $N$, namely, 
\begin{equation} \tag{$P_4$}   \label{P4.intro}
\lim_{r\to 0}\frac{N(2r) M(2r)^{p^-}}{r^n} =0,
\end{equation}
which allow us to prove in Theorem \ref{teo.comp} that the embedding $W^{M,N,G}(\R^n)\subset L^G_{loc}(\R^n)$ is compact 
(a generalization of the Rellich--Kondrachov theorem to our setting). 

Therefore, as in \cite{S}, we can be considered the Dirichlet eigenvalue problem  
\begin{align} \label{eig.prob}
\begin{cases}
(-\Delta_g)^{M,N} u = \lam g(u) &\quad \text{ in } \Omega\\
u=0 &\quad \text{ in } \R^n \setminus \Omega.
\end{cases}
\end{align}
Observe that in contrast with $p-$Laplacian type eigenproblems,  due to the possible lack of homogeneity, eigenfunctions depend strongly on the energy level: for any $\mu>0$ there exists $\lam_\mu>0$ and an
eigenfunction $u_\mu\in W^{M,N,G}_0(\Omega)$ normalized such that $\Phi_G(u_\mu)=\mu$, that is, 
$$
\lam_\mu(\Omega) \int_\Omega g(u_\mu) u_\mu \,dx  =  \langle (-\Delta_g)^{M,N} u_\mu,u_\mu\rangle,
$$
where $\langle\,,\,\rangle$ denotes the duality product and
$$
W^{M,N,G}_0(\Omega) := \left\{u\in W^{M,N,G}(\R^n) \colon u=0 \text{ a.e. in } \R^n\setminus \Omega \right\}.
$$

Consequently, the first eigenvalue of \eqref{eig.prob} can be defined as the less value  over all possible values of $\mu$, that is, 
$$
\lam_1^{M,N,G}(\Omega)=\inf\{\lam_\mu(\Omega):\mu>0\}.
$$

We also define  
$$
 \alpha_\mu(\Omega) = \inf \{ \Phi_{M,N,G}(u) : u\in W^{M,N,G}_0(\Omega) \text{ and } \Phi_G(u)=\mu\},
$$
and the best Poincar\'e constant over all possible values of $\mu$ is  denoted as
$$
\alpha_1^{M,N,G}(\Omega) = \inf\{ \alpha_\mu(\Omega) : \mu>0\}.
$$
It is worth of mention that although \eqref{P4.intro} is needed to define $\lam_1^{M,N,G}(\Omega)$, we can prescind from it to define $\alpha_1^{M,N,G}(\Omega)$.

As in \cite{S}, the quantities $\lam_1^{M,N,G}(\Omega)$ 
and $\alpha_1^{M,N,G}(\Omega)$ can be proved to be well defined.
Moreover,  since the spectrum of \eqref{eig.prob} is closed,  $\lam_1^{M,N,G}(\Omega)$ is an eigenvalue of  \eqref{eig.prob} as well.

However, a remarkable difference with the $p-$growth case lays on the fact that $\lam_1^{M,N,G}(\Omega)$ may differ from $\alpha_1^{M,N,G}(\Omega)$. In fact,  it holds that
$$
\frac{p^-}{p^+} \alpha_1^{M,N,G}(\Omega) \leq \lam_1^{M,N,G}(\Omega) \leq \frac{p^+}{p^-} \alpha_1^{M,N,G}(\Omega).
$$
In this context, in Corollaries \ref{cor1} and \ref{cor2} we provide for a Faber-Krahn type inequality for these constants. Namely,  for every $\Omega\subset \R^n$ open and bounded we have that
$$
\alpha_1^{M,N,G}(B)\leq \alpha_1^{M,N,G}(\Omega)
$$
where $B$ is a ball with $\mathcal{L}^n(B)=\mathcal{L}^n(\Omega)$. 
Furthermore, if in addition $G'$ also satisfies  condition \eqref{cond.intro}, then
$$
\lam_1^{M,N,G}(B)\leq \lam_1^{M,N,G}(\Omega)
$$
for any $B$ ball with $\mathcal{L}^n(B)=\mathcal{L}^n(\Omega)$.

\medskip

This paper is organized as follows. In section 
\ref{preliminares} we collect the basic definitions and properties of our fractional Orlicz-Sobolev spaces.  Section \ref{principio} is devoted to prove our main result.
In section \ref{section-compactness}, we prove our generalization of the the Rellich--Kondrachov compactness theorem to our setting. 
Finally in Section \ref{aplicaciones} we give some applications to the behavior of the Poincar\'e constant and the first eigenvalue of $(-\Delta_g)^{M,N}$ under symmetrization.


\section{Preliminaries on Fractional Orlicz-Sobolev spaces}\label{preliminares}
In this section we make a brief overview on the classical Orlicz-Sobolev spaces, as well as we introduce the general fractional order Orlicz-Sobolev spaces, their main properties  and the associated general fractional $g-$laplacian operator.

\subsection{Young functions}
By a Young function $G\colon\R \to \R$ we understand a function  fulfilling the following properties:
\begin{align*}
\tag{$G_0$}\label{G0} &G \text{ is even, continuous, convex,  increasing for $t>0$ and }  G(0)=0;\\
\tag{$G_1$}\label{G3} &  \lim_{x\to 0} \frac{G(x)}{x} = 0 \text{ and }\lim_{x\to \infty} \frac{G(x)}{x} = \infty.
\end{align*}
Denoting $g(t)=G'(t)$, if we assume that they are related through the following growth assumption
\begin{align*}
\tag{$G_2$}\label{cond}
1<p^-\leq \frac{tg(t)}{G(t)} \leq p^+ \quad \forall t>0,
\end{align*}
then, by \cite[Theorem 4.1]{Krasnoselskii} $G$ satisfies the $\Delta_2$ condition, i.e., 
\begin{align*}
\tag{$\Delta_2$}\label{D2} \text{there exists $\C>2$ such that } G(2t)\leq \C G(t), \text{ for any }t>0;
\end{align*}

Without loss of generality it can be assumed the normalization condition $G(1)=1$. This normalization will be assumed throughout the paper and will not be mentioned explicitly.

An immediate consequence of \eqref{cond} is the following polynomial growth, both on $G$ and $g$.
\begin{lema}\label{poliG}
Assume $G$ is an Young function satisfying \eqref{cond} and normalized as $G(1)=1$. Then we have the following polynomial growth
\begin{align}\label{cotaG1}
t^{p^-} &\le G(t) \le t^{p^+} \qquad \text{for } t>1 \\
\label{cotaG2}
t^{p^+} &\le G(t) \le t^{p^-} \qquad \text{for } 0<t<1\\
\label{cotag1}
p^- t^{p^- - 1} &\le g(t) \le p^+ t^{p^+ - 1} \qquad \text{for } t>1 \\
\label{cotag2}
p^- t^{p^+ - 1} &\le g(t) \le p^+ t^{p^- - 1} \qquad \text{for } 0<t<1.
\end{align}
\end{lema}

Moreover, the following holds.
\begin{lema}\label{lema.estimareg}
Let $G$ be an Young function satisfying \eqref{cond}. Then, for $0<t<1$ it holds that
\begin{equation}\label{estimareg}
\frac{p^-}{p^+} g(a) t^{p^+-1} \le g(a t)\le \frac{p^+}{p^-} g(a) t^{p^- - 1}.
\end{equation}
\end{lema}

The proofs of Lemmas \ref{poliG} and \ref{lema.estimareg} are elementary. The interested reader can find a proof in \cite{FBPLS}.

The complementary function of an Young function $G$ is defined as
$$
G^*(a) := \sup\{at-G(t)\colon t>0\}.
$$
From this definition  the following Young-type inequality holds
\begin{equation}\label{young}
at\le G(t) + G^*(a)\quad \text{for every } a,t\ge 0.
\end{equation}

It is easy to deduce the identity,
\begin{equation}\label{igualdad}
G^*(g(t)) = tg(t) - G(t),
\end{equation}
for every $t>0$, see \cite[Lemma 2.9]{FBS}. Now \eqref{igualdad} and \eqref{cond} yield that
\begin{equation}\label{cond*}
(p^+)'\le \frac{t g^*(t)}{G^*(t)}\le (p^-)',
\end{equation}
where $g^*(t) = (G^*)'(t)$. Observe that \eqref{cond*} implies that $G^*$ verifies the $\Delta_2$ condition. See \cite[Theorem 4.1]{Krasnoselskii}.

\begin{rem}\label{reflexivo}
 Notice that indeed, \cite[Theorem 4.1]{Krasnoselskii} entails that \eqref{cond} is equivalent to the fact that $G$ and $G^*$ both satisfy the $\Delta_2$ condition.
\end{rem}

\subsection{General Fractional Orlicz-Sobolev spaces} \label{sec.functiones.m.n}

In this subsection we study some basic properties of the space  $W^{M,N,G}(\R^n)$ defined in the introduction.

The spaces $L^G(\R^n)$ and $W^{M,N,G}(\R^n)$ are naturally endowed with the so-called Luxemburg norms
$$
\|u\|_G =  \|u\|_{L^G(\R^n)} := \inf\left\{\lambda>0\colon \Phi_G\left(\frac{u}{\lambda}\right)\le 1\right\}
$$
and
$$
\|u\|_{M,N,G} =  \|u\|_{W^{M,N,G}(\R^n)} := \|u\|_G + [u]_{M,N,G},
$$
where
$$
[u]_{M,N,G} :=\inf\left\{\lambda>0\colon \Phi_{M,N,G}\left(\frac{u}{\lambda}\right)\le 1\right\}.
$$

Let $G$ be an Young function such that $G'=g$ and $M$, $N$ satisfying properties \eqref{P1}--\eqref{P3}. The fractional $g-$Laplacian is defined as the gradient of the modular $\Phi_{M,N,G}\colon W^{M,N,G}(\R^n)\to \R$ and denoted by $(-\Delta_g)^{M,N}$.

This  operator is well defined and continuous between $W^{M,N,G}(\R^n)$ and its dual space $(W^{M,N,G}(\R^n))'$. Moreover, it has the expression
\begin{equation} \label{fracg}
\langle (-\Delta_g)^{M,N} u,v \rangle =  \int_{\R^{2n}} g\left( |D_M u| \right)\frac{D_M u}{|D_M u|} D_M v \,d\mu_N
\end{equation}
for any $u,v\in W^{M,N,G}(\R^n)$, where the \emph{$M-$H\"older quotient} is defined as
$$ 
D_M u(x,y) = \frac{u(x)-u(y)}{M(|x-y|)}. 
$$
and the measure
$$ 
d \mu_N = \frac{dx \; dy}{N(|x-y|)} 
$$
in defined in $\R^{2n}=\R^n \times \R^n$.

We start with the following basic properties on the usual $W^{1,G}$ and $L^G$ spaces.
\begin{prop}[\cite{Adams}, Chapter 8] \label{propiedades}
Let $G$ be an Young function satisfying \eqref{cond}. Then the spaces $L^G(\R^n)$ and $W^{1,G}(\R^n)$ are reflexive, separable Banach spaces. Moreover, the dual space of $L^G(\R^n)$ can be identified with $L^{G^*}(\R^n)$. Finally, $C_c^\infty(\R^n)$ is dense in $L^{G}(\R^n)$ and in $W^{1,G}(\R^n)$.
\end{prop}

\noindent In order to deduce similar properties for our general fractional Orlicz-Sobolev, following \cite{Raman}, we observe that the mapping
$$
\mathcal{A}\colon W^{M,N,,G}(\R^n) \to X= L^G(\R^n) \times L^{G}(\R^{2n},\mu_N) 
$$
given by $\mathcal{A} = (u, D_M u)$ is an isometry. 

For later purpose, it is convenient to observe that the measure $\mu_N$ is invariant under the action of the diagonal translation operator:
$$ 
\tau_z \phi (x,y) = \phi (x-z,y-z). 
$$
Moreover, the $M-$H\"older quotient $D_M$ commutes with $\tau_z$, i.e.
$$
\tau_z D_M u (x,y) = D_M (\tau_z u) (x,y),
$$
where, as usual, $\tau_z u(x) = u(x-z)$.

\begin{prop}
$W^{M,N,G}(\R^n)$ is a separable Banach space. If furthermore, $G$ satisfies \eqref{cond}, it is reflexive.
\end{prop}

\begin{proof}
The proposition will follows if we show that the image of $\mathcal{A}$ is closed, hence $W^{M,N,G}(\R^n)$ is isometrically isomorphic to a closed subspace of the reflexive Banach space $X$ and the result follows.

In order to do so, assume that $(u_k,D_M u_k) \to (u,\phi)$ in  $X$. Then passing to a subsequence, we may assume that  $u_k \to u$ and $D_M u_k \to \phi$ a. e. 

Hence $D_M u_k \to D_M u$ a.e., whence $\phi= D_M u$.
\end{proof}

We prove that the general fractional Orlicz Sobolev space contains $W^{1,G}(\Omega)$ as a subspace.

\begin{lema} \label{teo1}
Let $u\in W^{1,G}(\R^n)$. Then, for $0<s<1$ it holds that
$$
\Phi_{M,N,G}(u) \leq C\left( \Phi_{G}(|\nabla u|)  +  \Phi_{G}(u)\right) 
$$
where $C$ depends on $G$, $M$, $N$ and $n$.
\end{lema}

\begin{proof}
Let us first assume that $u\in C^2_c(\R^n)$.

We split the integral
\begin{align*}
 \int_{\R^{2n}} &G\left( \frac{|u(x)-u(y)|}{M(|x-y|)} \right) \frac{dx\,dy}{N(|x-y|)}\\ &=
 \left(\int_{B_1}\intr + \int_{B_1^c} \intr \right)  G\left( \frac{|u(x+h)-u(x)|}{M(|h|)} \right) \frac{dx\,dh}{N(|h|)} \\
 &:=I_1+I_2.
\end{align*}

Let us bound $I_1$. Given $u\in C^2_c(\R^n)$, observe that for any fixed $x\in \R^n$ and $h\in\R^n$ we can write
$$
 u(x+h)-u(x) = \int_0^1 \frac{d}{dt} u(x+th) \,dt = \int_0^1 \nabla u(x+th)\cdot h \,dt.
$$
Dividing by $M(|h|)$ and using the monotonicity and   convexity of $G$ we get
\begin{align} \label{ec.lema.1}
\begin{split}
 G\left( \frac{|u(x+h)-u(x)|}{M(|h|)} \right) 
 &\leq 
 G \left(  \int_0^1 |\nabla u(x+th)| \frac{|h|}{M(|h|)} \,dt \right)\\
 &\leq
    \int_0^1 G \left(|\nabla u(x+th)| \frac{|h|}{M(|h|)} \right)dt.  
\end{split}    
\end{align}
Expression \eqref{ec.lema.1} together with \eqref{cond}, \eqref{P2} and \eqref{P3} allow us to bound $I_1$ as follows
\begin{align*}
I_1&\leq \int_{B_1 }\intr \int_0^1 G \left(|\nabla u(x+th)| \frac{|h|}{M(|h|)} \right)dt\, dx \,\frac{dh}{N(|h|)}\\
&\leq \int_{B_1 } \frac{|h|^{p^-}}{N(|h|)M(|h|)^{p^-}}\intr \int_0^1 G \left(|\nabla u(x+th)|\right)dt\, dx \,dh\\
&\leq n\omega_n \int_0^1   \frac{r^{p^-+n-1}}{N(r)M(r)^{p^-}} \,dr    \intr  G \left(|\nabla u(x)|\right)  dx.
\end{align*}

The integral $I_2$ can be bounded using again \eqref{cond}, \eqref{P2} and \eqref{P3}. Indeed, 
\begin{align*}
 I_2&\leq \int_{B_1^c}\frac{1}{N(|h|)M(|h|)^{p^-} } \intr G(|u(x+h)-u(x)|) \,dx \,dh\\
&\leq \C  \int_{B_1^c}\frac{1}{N(|h|)M(|h|)^{p^-} } \intr G(|u(x+h)|) + G(|u(x)|) \,dx \,dh\\ 
&\leq 2n\omega_n \C \int_1^\infty \frac{r^{n-1}\,dr}{N(r) M(r)^{p^-} }  \intr G(|u(x)|)\,dx.
\end{align*}

In order to prove the Lemma for any $u\in W^{1,G}(\R^n)$, we take a sequence $\{u_k\}_{k\in\N}\subset C^2_c(\R^n)$ such that $u_k\to u$ in $W^{1,G}(\R^n)$. Without loss of generality, we may assume that $u_k\to u$ a.e. in $\R^n$. Observe that this implies that
$$
G\left(\frac{|u_k(x)-u_k(y)|}{M(|x-y|)}\right)\to G\left(\frac{|u(x)-u(y)|}{M(|x-y|)}\right) \quad \text{a.e. in } \R^n\times\R^n.
$$

Therefore, by   Fatou's Lemma, we obtain that
\begin{align*}
\Phi_{M,N,G}(u)&\le \liminf_{k\to\infty} \Phi_{M,N,G}(u_k)\\
& \le \lim_{k\to\infty} C\left(\Phi_{G}(|\nabla u_k|) + \Phi_{G}(u_k)\right)
=  C\left(\Phi_{G}(|\nabla u|) + \Phi_{G}(u) \right).
\end{align*} 
The proof is now complete.
\end{proof}

We finish this section stating a Poincar\'e type inequality for functions in $W^{M,N,G}_0(\Omega)$ from where we conclude  that $[\,\cdot\,]_{M,N,G}$ is an equivalent norm to $\|\cdot\|_{M,N,G}$ in $W^{M,N,G}_0(\Omega)$.

The proof is completely analogous to \cite[Theorem 2.12]{FBPLS} and it is omitted.
\begin{thm}\label{thm.poincare}
Let $\Omega\subset \R^n$ be open and bounded. Then, there exists a constant $C$ depending on $M, N$ and the diameter of $\Omega$ such that 
$$
\Phi_G(u) \le  \Phi_{M,N,G}(C u),
$$
for every $u\in W^{M,N,G}_0(\Omega)$.
\end{thm}

As a corollary we infer the following Poincar\'e's  inequality for fractional Luxemburg type norms. The proof is identical to that of \cite[Corollary 2.13]{FBPLS}.
\begin{cor} \label{poincare.norma}
Let $\Omega\subset \R^N$ be open and bounded. Then
$$
\|u\|_{G} \leq C [u]_{M,N,G}
$$
for every $0<s<1$ and $u\in W^{M,N,G}_0(\Omega)$, where $C$ depends on the diameter of $\Omega$, $n,  p^+$ and $p^-$.
\end{cor}

\subsection{Examples} 
\label{subsection-examples}
\begin{itemize}
\item When $M(r)=r^s$, $0<s<1$ and $N(r)=r^n$, $n\geq 1$ we obtain the fractional Orlicz-Sobolev spaces defined in \cite{FBS}. Indeed, \eqref{P1} and \eqref{P2} easily hold and \eqref{P3} is fulfilled since 
$$
\int_0^1 \frac{r^{n-1+p^-}}{N(r)M(r)^{p^-}} \,dr =\frac{1 }{p^-(1-s)}, \qquad    \int_1^\infty \frac{r^{n-1}}{N(r)M(r)^{p^-}}\,dr =  \frac{1}{sp^-}.
$$

\item When $M(r)=r$ and $N(r)=r^{n-1}$, $n\geq 1$ we obtain the Orlicz-Slobodetskii spaces defined in \cite{KK}. Indeed, \eqref{P1} and \eqref{P2} easily hold and \eqref{P3} reads as
$$
  \int_0^1 \frac{r^{n-1+p^-}}{N(r)M(r)^{p^-}} \,dr =1 , \qquad    \int_1^\infty \frac{r^{n-1}}{N(r)M(r)^{p^-}}\,dr =  \frac{1}{p^--1}.
$$

\item When $M(r)=r^s (1+|\log r|)^\beta$, $\beta>0$,  and $N(r)=r^n$  we obtain the family of weighted Besov Spaces considered in \cite{A}. Indeed, \eqref{P1} and \eqref{P2} easily hold and \eqref{P3} is fulfilled since 
$$
\int_0^1 \frac{r^{n-1+p^-}}{N(r)M(r)^{p^-}} \,dr \leq \int_0^1 r^{p^-(1-s)-1}\,dr =\frac{1 }{p^-(1-s)}   
$$
$$
\int_1^\infty \frac{r^{n-1}}{N(r)M(r)^{p^-}}\,dr \leq  \int_1^\infty r^{-1-sp^-}\, dr =  \frac{1}{sp^-}
$$
 
\item When $M(r)=r^sG^{-1}(r^n)$ and $N(r)=1$, $n\geq 1$ we obtain the Orlicz-Slobodetskii  spaces defined in \cite{ABS}.  Indeed, \eqref{P1} and \eqref{P2} easily hold and \eqref{P3} reads as
$$
\int_0^1 \frac{r^{n-1+p^-}}{N(r)M(r)^{p^-}} \,dr \leq \int_0^1 r^{p^-(1-s)-1}\,dr =\frac{1 }{p^-(1-s)},   
$$
$$
\int_1^\infty \frac{r^{n-1}}{N(r)M(r)^{p^-}}\,dr \leq  \int_1^\infty r^{n\left(1-\frac{p^-}{p^+}\right) -1-sp^-}\, dr <\infty
$$
if and only if 
$$
\frac{n}{s}< \frac{(p^-)^2}{p^+-p^-},
$$
where we have used that $G^{-1}(r)\ge \min\{r^{1/p^+}, r^{1/p^-}\}$.
\end{itemize}
 
\subsection{The density theorem} \label{sec.density}

In this subsection we show that test functions are dense in $W^{M,N,G}(\R^n)$. Even though we use the standard method of truncation and regularization, we remark that there are some subtle technicalities in the argument which lead us to write them down in detail. Indeed, a detailed proof seems to be missed in the previous works on the subject in the literature.

\begin{prop}\label{prop.test.densas}
Let $G$ be an Young function satisfying \eqref{cond} and $M$ and $N$ fulfilling \eqref{P1}--\eqref{P3}. 
Then $C_c^\infty(\R^n)$ is dense in $W^{M,N,G}(\R^n)$.
\end{prop}

As usual, we denote by $\rho\in C^\infty_c(\R^n)$ the standard mollifier with $\supp(\rho)=B_1(0)$ and $\rho_\ve(x)=\ve^{-n}\rho(\tfrac{x}{\ve})$ is the approximation of the identity. It follows that $\{\rho_\ve\}_{\ve>0}$ is a family of positive functions satisfying 
$$
\rho_\ve\in C_c^\infty(\R^n), \quad \supp(\rho_\ve)=B_\ve(0), \quad \intr \rho_\ve\,dx=1.
$$
Given $u\in L^G(\R^n)$ we define the regularized functions $u_\ve\in L^G(\R^n)\cap C^\infty(\R^n)$ as
\begin{equation} \label{regularizada}
u_\ve(x)=u*\rho_\ve(x).
\end{equation}

Moreover, given $\phi\in L^G(\R^{2n}, d\mu_N)$, we define 
\begin{equation}\label{regularizada.phi}
\phi_\ve(x,y) := \intr \phi(x-z,y-z)\rho_\ve(z)\, dz,
\end{equation}
and observe that if $u\in W^{M,N,G}(\R^n)$, then $D_M u_\ve(x,y) = (D_M u)_\ve(x,y)$. Also, if $\phi, \psi\in L^G(R^{2n}, d\mu_N)$, then $(\phi + \psi)_\ve = \phi_\ve + \psi_\ve$.

In this context we prove the following useful estimate on regularized functions.
\begin{lema} \label{lema.reg}
Let $\phi \in L^G(\R^{2n}, d\mu_N)$ and $\{\phi_\ve\}_{\ve>0}$ be the family defined in \eqref{regularizada.phi}. Then
$$
\int_{\R^{2n}} G(|\phi_\ve|)\, d\mu_N \leq  \int_{\R^{2n}} G(|\phi|)\, d\mu_N
$$
for all $\ve>0$.
\end{lema}

\begin{proof}
By Jensen's inequality
\begin{align*}
 G\left( |\phi_\ve(x,y)|  \right) &=  G\left( \left|\intr \phi(x-z, y-z) \rho_\ve(z)\,dz \right| \right)\\
 &\leq  \intr  G\left( |\phi(x-z, y-z)|\right) \rho_\ve(z)\,dz.
\end{align*}

Integrating the last inequality over the whole $\R^{2n}$ with respect to the measure $\mu_N$ we get
\begin{align*}
\int_{\R^{2n}} G\left( |\phi_\ve(x, y)|\right) d\mu_N(x,y) &\leq 
\int_{\R^{2n}} \left\{ \intr  G\left( |\phi(x-z, y-z)|\right) \rho_\ve(z)\,dz \right\}  d\mu_N(x,y) \\
&= \intr \left\{ \int_{\R^{2n}}  G\left(|\phi(x-z, y-z)|\right) d\mu_N(x,y) \right\} \rho_\ve(z)\,dz   \\
&=
\int_{\R^{2n}}  G\left(|\phi(x, y)|\right) d\mu_N(x,y),
\end{align*}
where we have used the invariance of the measure $d\mu_N$ with respect to the action of $\tau_z$ and the fact that $\intr \rho \,dz =1$.
\end{proof}
As an immediate corollary, we obtain:
\begin{cor}
Let $u\in W^{M,N,G}(\R^n)$. Then
$$
\Phi_{M,N,G}(u_\ve)\le \Phi_{M,N,G}(u).
$$
\end{cor}
\medskip
We set:
$$ 
\R^{2n}_* = \{ (x,y) \in \R^{2n} \colon x \neq y \}.
$$

By the change of variable $h=x-y$ and using polar coordinates,  we see that \eqref{P1} implies that $\mu_N(K)< \infty$ for every compact set $K \subset \R^{2n}_*$. This means that $\mu_N$ is a Radon measure on $\R^{2n}_*$. Note that in general this is not true for compact subsets of $\R^{2n}$.

\begin{lema}
$L^G(\R^{2n},\mu_N)\cap L^\infty(\R^{2n})$ is dense in $L^G(\R^{2n},\mu_N)$.
\end{lema}

\begin{proof}
Given $k>0$ and $\phi\in L^G(\R^{2n},\mu_N)$, we define
$$
\phi_k(x,y) := \begin{cases}
k & \text{if } \phi(x,y)\ge k\\
\phi(x,y) & \text{if } -k<\phi(x,y)<k\\
-k & \text{if } \phi(x,y)\le -k.
\end{cases}
$$
Observe that $|\phi_k|\le \min\{|\phi|, k\}$, hence, $\phi_k\in L^G(\R^{2n},\mu_N)\cap L^\infty(\R^{2n})$. Also, $\phi_k\to\phi$ a.e. as $k\to\infty$. Finally, observe that
$$
G(|\phi-\phi_k|)\le G(2|\phi|)\in L^1(\R^{2n}, \mu_N).
$$
So, we can apply the Lebesgue Dominated Convergence Theorem to conclude that
$$
\int_{\R^{2n}} G(|\phi-\phi_k|)\, d\mu_N\to 0\quad \text{as } k\to\infty,
$$
and the proof is completed.
\end{proof}

\begin{lema}\label{cont.densa}
$C_c(\R^{2n}_*)$ is dense in $L^G(\R^{2n},\mu_N)$. 
\end{lema}

\begin{proof}
Consider $\phi \in L^G(\R^{2n},\mu_N)$. We may assume that $\phi\in L^\infty(\R^{2n})$ and that
$$ 
\int_{\R^{2n}} G(|\phi|) \; d\mu_N \leq 1. 
$$
Given $\varepsilon >0$, we can choose $\delta>0$ such that
$$ 
\int_{\R^{2n}-K_\delta} G(|\phi|) \, d\mu_N < \frac{\varepsilon}{2}
$$ 
where
$$ 
K_\delta = \{ (x,y) \in \R^{2n} : |x-y| \geq \delta,|x| \leq 1/\delta,|y| \leq 1/\delta\}. 
$$
Using then Lusin's theorem (\cite[theorem 2.23 of chapter II]{Rudin}), which we may since $\mu_N$ is a Radon measure in $\R^{2n}_*$, we can construct a function $\psi \in C_c(\R^{2n}_*)$ such that $\mu_{N}(\{ (x,y)\in K_\delta: \phi(x,y) \neq \psi(x,y) \})< \frac{\varepsilon}{2G(2\|\phi\|_\infty)}$ and $\|\psi\|_\infty\le \|\phi\|_\infty$. Then
$$ 
\int_{K_\delta} G(|\phi-\psi|) \, d\mu_N \leq G(2\|\phi\|_\infty) \mu_N(\{ (x,y)\in K_\delta: \phi(x,y) \neq \psi(x,y) \}) \leq \frac{\varepsilon}{2}. 
$$  
Hence
$$  \int_{\R^{2n}_*} G(|\phi-\psi|) \, d\mu_N < \varepsilon $$
which implies the desired result.
\end{proof}

With this preliminaries we can now conclude the following result:
\begin{prop}
Let $\phi\in L^G(\R^{2n}, d\mu_N)$ and $\{\phi_\ve\}_{\ve>0}$ defined in \eqref{regularizada.phi}. Then
$$
\|\phi-\phi_\ve\|_{L^G(\R^{2n}, d\mu_N)}\to 0 \quad \text{as } \ve\to 0.
$$
\end{prop}

\begin{proof}
Let $\phi\in L^G(\R^{2n}, d\mu_N)$ and $\delta>0$. Then, using Lemma \ref{cont.densa}, there exists $\psi\in C_c(\R^{2n}_*)$ such that $\|\phi-\psi\|_{L^G(\R^{2n}, d\mu_N)}<\delta$.

Let $\{\phi_\ve\}_{\ve>0}$ and $\{\psi_\ve\}_{\ve>0}$ be the regularized functions given by \eqref{regularizada.phi}. Observe that since $\psi\in C_c(\R^{2n}_*)$, it follows that the support of $\phi_\ve$ is compact and bounded uniformly in $\ve>0$. Moreover, $\psi_\ve\to\psi$ uniformly. These immediately imply that
$$
\|\psi-\psi_\ve\|_{L^G(\R^{2n}, d\mu_N)}\to 0,
$$
as $\ve\to 0$.

Therefore, using Lemma \ref{lema.reg},
\begin{align*}
\|\phi - \phi_\ve\|&_{L^G(\R^{2n}, d\mu_N)}\\
&\le \|\phi - \psi\|_{L^G(\R^{2n}, d\mu_N)}+ \|\psi - \psi_\ve\|_{L^G(\R^{2n}, d\mu_N)}+\|\psi_\ve - \phi_\ve\|_{L^G(\R^{2n}, d\mu_N)}\\
&\le 2\delta +  \|\psi - \psi_\ve\|_{L^G(\R^{2n}, d\mu_N)}
\end{align*}
and the result follows.
\end{proof}

From this result we get:
\begin{cor}\label{reg.converge}
Let $u\in W^{M,N,G}(\R^n)$ and let $\{u_\ve\}_{\ve>0}$ be the regularized functions defined in \eqref{regularizada}. Then
$$
\|u-u_\ve\|_{M,N,G}\to 0, \quad \text{as }\ve\to 0.
$$
\end{cor}

We also need estimates on modulars of truncated functions. We use the following notations: Let $\eta\in C_c^\infty(\R^n)$ such that $\eta=1$ in $B_1(0)$, $\supp (\eta)=B_2(0)$, $0\leq \eta\leq 1$ in $\R^n$ and $\|\nabla \eta\|_\infty\le 2$. Given $k\in\N$ we define $\eta_k(x)=\eta(\tfrac{x}{k})$. Observe that  $\{\eta_k\}_{k\in\N} \in C_c^\infty(\R^n)$ and
$$
0\leq \eta_k \leq 1, \quad \eta_k =1 \text{ in } B_k(0), \quad  \supp (\eta_k)=B_{2k} (0),\quad  |\nabla \eta_k|\leq \frac{2}{k}.
$$
Given $u\in L^G(\R^n)$ we define the truncated functions $u_k$, $k\in\N$ as 
\begin{equation} \label{truncada}
u_k=\eta_k u.
\end{equation}

In the next lemma we analyze the behavior of the modular of truncated functions. 

\begin{lema} \label{lema.trunc}
Let $u\in W^{M,N,G}(\R^n)$ and $\{u_k\}_{k\in\N}$ be the functions defined in \eqref{truncada}. Then
$$
\Phi_{M,N,G}(u_k) \leq C(\Phi_{M,N,G}(u) + \Phi_G (u) )
$$
where $C>0$ is independent of $u$.
\end{lema}

\begin{proof}
From \eqref{cond} and since $\eta_k\leq 1$ we have
\begin{align*}
G\left(\frac{|u_k(x)-u_k(y)|}{ M(|x-y|)}\right) &\leq  C\left(G\left(\frac{|u(x)-u(y)|}{ M(|x-y|)}\right) + G\left(\frac{|u(x)||\eta_k(x)-\eta_k(y)|}{ M(|x-y|)}\right)\right).
\end{align*}
Then we get
\begin{align*}
\int_{\R^{2n}} G &\left(  \frac{|u_k(x)-u_k(y)|}{M(|x-y|)} \right) d\mu_N(x,y)  \leq \\
&C \Phi_{M,N,G}(u) + C \int_{\R^{2n}}  G\left(\frac{|u(x)||\eta_k(x)-\eta_k(y)|}{ M(|x-y|)}\right) d\mu_N(x,y).
\end{align*}

The integral above can be splited as follows.
\begin{align*}
\left(\intr \int_{|x-y|\geq 1} + \intr \int_{|x-y|< 1}\right)  G\left(\frac{|u(x)||\eta_k(x)-\eta_k(y)|}{ M(|x-y|)}\right) d\mu_N(x,y) := I_1 + I_2.
\end{align*}

The monotonicity of $G$ and \eqref{cotaG2} allow us to bound $I_1$ as follows
\begin{align*}
I_1 &\leq
\intr\int_{|x-y|\geq 1}  \frac{G(2|u(x)|)}{M(|x-y|)N(|x-y|)}\,dx\,dy\\
&= \int_{B_1^c} \frac{dh}{M(|h|) N(|h|)} \intr G(2|u(x)|) \,dx\\
&\leq C \Phi_G(u) .
\end{align*}

We deal now with $I_2$. Using that $|\nabla \eta_k|\leq \tfrac{2}{k}$,   \eqref{P2} and \eqref{P3}, we get
\begin{align*}
\intr \int_{|x-y|<1} G&\left(\frac{|u(x)||\eta_k(x)-\eta_k(y)|}{ M(|x-y|)}\right)d\mu_N(x,y) \leq \\
&\intr \int_{|x-y|\leq 1} G\left(\frac{2}{k} \frac{|u(x)| |x-y|}{  M(|x-y|)}\right) d\mu_N(x,y)\\
&\leq C \int_{B_1} \frac{|h|^{p^-}}{M(h)^{p^-}N(h)} \,dh \Phi_G(u)\\
&=C \Phi_G(u).
\end{align*}
From these estimates the conclusion of the lemma follows.
\end{proof}

\begin{lema}\label{lema.trunc}
Let $u\in W^{M,N,G}(\R^n)$ and $\{u_k\}_{k\in\N}$ be the functions defined in \eqref{truncada}. Then
$$
\| u_k - u\|_{M,N,G} \to 0 \quad \text{ as } k\to\infty.
$$
\end{lema}

\begin{proof}
Observe that $u_k\to u$ a.e. and $|u_k-u|\leq 2|u|\in L^G(\R^n)$, then by the dominated convergence theorem we have that $\Phi_G(u-u_k)\to 0$. For the second part, since
$$
D_M u_k(x,y)= u(x) D_M \eta_k (x,y) + \eta_k(y) D_M u(x,y)\quad a.e.
$$
we have that
$$
|D_M u_k(x,y)| \leq  2\frac{|u(x)||x-y|}{M(|x-y|)}\chi_{B(x,1)}(y) +   2\frac{|u(x)|}{M(|x-y|)}\chi_{B(x,1)^c}(y)  + |D_M u(x,y)|
$$
and from condition \eqref{P3}, it follows easily that the right hand side of the inequality above belongs to $L^G(\R^{2n}, d\mu_N)$. Therefore, using again the dominated convergence theorem for Orlicz spaces we get that $\Phi_{M,N,G}(u-u_k) \to 0$.
\end{proof}

Finally, after all these preparatives, the proof of Proposition \ref{prop.test.densas} follows immediately.
\begin{proof}[Proof of Proposition \ref{prop.test.densas}]
At this point, the proof is a simple combination of Corollary \ref{reg.converge} and Lemma \ref{lema.trunc}.
\end{proof}

\section{The P\'{o}lya-Szeg\"{o} principle} \label{principio}
In this section we prove our main result. Throughout this section $H\subset \R^n$ will denote a closed half-space and $\tilde x$ the reflexion respect to $\partial H$. 

Let us begin with the definition of {\em polarization}. Polarization is a useful tool for proving and studying rearrangements.

\begin{defn}
If $H\subset \R^n$ is a closed half-space, $\sigma_H\colon \R^n\to \R^n$ is the reflexion with respect to $\partial H$, and $u\colon \R^n\to \R$ is a measurable function, the \emph{polarization} of $u$ with respect to $H$ is the function $u_H\colon \R^n\to \R$ defined by
\begin{align}  \label{polarizada}
u_H(x):=
\begin{cases}
\max\{u(x), u(\tilde x)\} &\quad \text{ if } x \in H\\
\min\{u(x), u(\tilde x)\} &\quad \text{ if } x \in \R^n\setminus H.
\end{cases}
\end{align}
\end{defn}

\begin{rem} The following trivial inequalities will be most useful in the sequel. Namely, for any  $x,y\in H$ it holds that
\begin{equation} \label{rel.tilde}
|x-y|=|\tilde x-\tilde y|, \qquad |\tilde x-y|= |x-\tilde y|,
\end{equation}
\begin{equation} \label{rel.triang}
|x-y|\leq |x-\tilde y|.
\end{equation}
\end{rem}

The following inequality concerning two points rearrangements is useful for our computations.

\begin{lema}\cite[Lemma 1]{ChPu} \label{lema.elemental}
Let $G\colon \R^+\to \R^+$ be a strictly increasing convex function. Then, for all real numbers $\alpha,\beta,\gamma,\delta\in\R$ such that $\gamma<\alpha$ and $\delta<\beta$, it holds that
$$
G(\lam|\alpha-\beta|)+G(\lam|\gamma-\delta|)\leq G(\lam|\alpha-\delta|)+G(\lam|\gamma-\beta|)
$$
where $\lam>0$ is a fixed number.
\end{lema}

With the help of the following two lemmas we prove that polarized functions have always modulars less than the original ones.

\begin{lema} \label{lema1}
Let $G$ be a increasing convex function and $M$ be a positive function. Assume that $x,y\in H$. Then it holds that
\begin{align*} 
G\left( \frac{u_H(x) - u_H(\tilde y)}{M(|x-\tilde{y}|)} \right) + G\left( \frac{u_H(\tilde x)  - u_H(y)}{M(|x-\tilde y|)} \right)  
\leq G\left( \frac{u(x) - u(\tilde y)}{M(|x-\tilde{y}|)} \right) + G\left( \frac{u(\tilde x)  - u(y)}{M(|x-\tilde y|)} \right) 
\end{align*}
\end{lema}

\begin{proof}
Consider $x,y\in H$. When $u(x)\geq u(\tilde x)$ and $u(y)\geq u(\tilde y)$ using \eqref{polarizada} we get
$$
u_H(x)-u_H(\tilde y)= u(x)-u(\tilde y), \qquad u_H(\tilde x)-u_H(y)= u(\tilde x)-u(y)
$$
and the result follows. When $u(x)\geq u(\tilde x)$ and $u(y) \leq u(\tilde y)$  we have that
$$
G\left( \frac{u_H(x) - u_H(\tilde y)}{M(|x-\tilde{y}|)} \right) + G\left( \frac{u_H(\tilde x)  - u_H(y)}{M(|x-\tilde y|)} \right)  = 
G\left( \frac{u(x)-u(y)}{M(|x-\tilde{y}|)} \right) + G\left( \frac{u(\tilde x)  - u(\tilde y)}{M(|x-\tilde y|)} \right)  
$$
then applying Lemma \ref{lema.elemental} with $\lam=M(|x-\tilde y|)$, $\alpha=u(x)$, $\beta=u(y)$, $\gamma=u(\tilde x)$, $\delta=u(\tilde y)$ we have that
$$
G\left( \frac{u(x) - u(y)}{M(|x-\tilde{y}|)} \right) +  G\left( \frac{u(\tilde x) - u(\tilde y)}{M(|x-\tilde{y}|)} \right) \leq 
 G\left( \frac{u(x) - u(\tilde y)}{M(|x-\tilde{y}|)} \right) + G\left( \frac{u(\tilde x)  - u(y)}{M(|x-\tilde y|)} \right) 
$$
and the lemma follows.

The remaining cases follow analogously.
\end{proof}

\begin{lema} \label{lema2}
Let $G$ be a increasing convex function and $M$ be a positive function. Assume that $x,y\in H$. Then it holds that
\begin{multline*} 
G\left( \frac{u_H(\tilde x) - u_H(\tilde y)}{M(| x- y|)} \right)   +G\left( \frac{u_H(x) - u_H(y)}{M(|x-y|)} \right)    \leq  
G\left( \frac{u(x) - u(y)}{M(|x-y|)} \right) + G\left( \frac{u(\tilde x) - u(\tilde y)}{M(|x-y|)} \right).
\end{multline*}

\end{lema}

\begin{proof}
Consider $x,y\in H$. When $u(x)\geq u(\tilde x)$  and  $u(y)\geq u(\tilde y)$, from  \eqref{polarizada} we get
$$
u_H(x)-u_H(y)= u(x)-u(y), \qquad u_H(\tilde x)-u_H(\tilde y)= u(\tilde x)-u(\tilde y)
$$
and the result follows. On the another hand, if we assume that $u(x)\geq u(\tilde x)$ and $u(y)\leq u(\tilde y)$, from \eqref{polarizada} we have that
$$
u_H(x)-u_H(y)= u(x)-u(\tilde y), \qquad u_H(\tilde x)-u_H(\tilde y)= u(\tilde x)-u(y).
$$
In this case, the lemma follows applying Lemma \ref{lema.elemental} with $\lam=M(|x-y|)$, $\alpha=u(x)$, $\beta=u(\tilde y)$, $\gamma=u(\tilde x)$, $\delta=u(y)$.

The remaining cases follow analogously.
\end{proof}

\begin{prop}\label{one-reflection}
Let $G$ be a increasing convex function and $M,N$ satisfying \eqref{P1}--\eqref{P3}. Then we have that
$$
\Phi_{M,N,G}(u_H) \leq  \Phi_{M,N,G}(u).
$$
\end{prop}

\begin{proof}
We split $\R^n\times \R^n$ into the four regions $H\times H$, $H^c\times H^c$, $H\times H^c$, $H^c\times H$. Using \eqref{rel.tilde} we have that
\begin{align*}
\iint_{H^c\times H^c} G\left(D_M u_H\right)\, d\mu_N & =
\iint_{H\times H} G\left(\frac{u_H(\tilde x)-u_H(\tilde y)}{M(| x- y|)}  \right)\frac{dx\,dy}{N(| x-y|)}\\
\iint_{H\times H^c} G\left(D_M u_H\right)\, d\mu_N& =
\iint_{H\times H} G\left(\frac{u_H(x)-u_H(\tilde y)}{M(|x-\tilde y|)}  \right)\frac{dx\,dy}{N(| x-\tilde y|)}\\
\iint_{H^c\times H} G\left(D_M u_H\right)\, d\mu_N & =
\iint_{H\times H} G\left(\frac{u_H(\tilde x)-u_H(y)}{M(|x-\tilde y|)}  \right)\frac{dx\,dy}{N(| x-\tilde y|)},
\end{align*}
and hence, the expression $\Phi_{M,N,G}(u_H)$ can be written as
\begin{align*}
\iint_{H\times H} &\Big[ G\left(\frac{u_H(x)-u_H(y)}{M(|x-y|)}  \right)\frac{1}{N(|x-y|)} + G\left(\frac{u_H(\tilde x)-u_H(\tilde y)}{M(|x-y|)}  \right)\frac{1}{N(|x-y|)}\\
&+ G\left(\frac{u_H(\tilde x)-u_H(y)}{M(|x-\tilde y|)}  \right)\frac{1}{N(|x-\tilde y|)} +G\left(\frac{u_H(x)-u_H(\tilde y)}{M(|x-\tilde y|)}  \right)\frac{1}{N(|x-\tilde y|)} \Big] dx\,dy.
\end{align*}
Now, applying Lemma \ref{lema2} to the first two terms terms in the last expression and Lemma  \ref{lema1} to the last two ones, we can bound the above expression as
\begin{align*}
\iint_{H\times H} &\Big[ G\left(\frac{u(x)-u(y)}{M(|x-y|)}  \right)\frac{1}{N(|x-y|)} + G\left(\frac{u(\tilde x)-u(\tilde y)}{M(|x-y|)}  \right)\frac{1}{N(|x-y|)}\\
&+ G\left(\frac{u(\tilde x)-u(y)}{M(| x-\tilde y|)}  \right)\frac{1}{N(| x-\tilde y|)} +G\left(\frac{u(x)-u(\tilde y)}{M(|x-\tilde y|)}  \right)\frac{1}{N(|x-\tilde y|)} \Big] dx\,dy,
\end{align*}
which is precisely $\Phi_{M,N,G}(u)$.
\end{proof}

Finally, using the construction provided in \cite{VS} together with the previous proposition, we prove our main result.

\begin{thm} \label{main}
Let $G$ be a increasing convex function and $M, N$ satisfying \eqref{P1}--\eqref{P3}. If $u \in W^{M,N,G}(\R^n)$, then $u^* \in W^{M,N,G}(\R^n)$ and
$$  \Phi_{M,N,G}(u^*) \leq \Phi_{M,N,G}(u) $$ 
\end{thm}

\begin{proof}

Given $u \in W^{M,N,G}(\R^n)$, define $u_k$ as 
\begin{align*}
\begin{cases}
&u_0=u,\\
&u_{k+1}=(u_k)_{H_1\cdots H_{k+1}},
\end{cases}
\end{align*}
where $\{H_k\}_{k\in\N}$ is a dense sequence in the set of closed half-spaces of which $0$ is an interior point.

Observe that  by Proposition \ref{prop.test.densas}, $u$ belongs to the closure of $C_c^\infty(\R^n)$, hence by  \cite[Section 4.1]{VS} it holds that
$$ 
u_k \to u^* \; \hbox{strongly in} \; L^G(\R^n) 
$$
and passing to a subsequence we will have
$$ u_{k_j} \to u^* \; \hbox{a.e.} $$
The continuity of $G$ implies that
$$ 
G\left(D_M u_{k_j}(x,y)\right)\to G\left(D_M u^*(x,y)\right)\; \hbox{a.e. in } \; \R^{2n}.
$$
Moreover, from Proposition \ref{one-reflection}, 
$$ 
\Phi_{M,N,G}(u_{k_j}) \leq \Phi_{M,N,G}(u). 
$$
Hence, by Fatou's lemma, It follows that 
$$ 
\Phi_{M,N,G}(u^*) \leq   \Phi_{M,N,G}(u),
$$
as we wanted to show.
\end{proof}

A P\'{o}lya-Szeg\"{o} principle for norms can be easily deduced from the previous result.

\begin{cor}
Let $G$ be a increasing convex function, $M$, $N$ satisfying \eqref{P1}--\eqref{P3} and $u \in W^{M,N,G}(\R^n)$. Then 
$$
[u^*]_{M,N,G} \leq [u]_{M,N,G}.
$$
Moreover, 
$$
\|u^*\|_{M,N,G} \leq \|u\|_{M,N,G}.
$$
\end{cor}
\begin{proof}
Given $u\in W^{M,N,G}(\R^n)$,  applying Theorem \ref{main} to the function $u/[u]_{M,N,G}$ and according to the definition of $[\cdot]_{M,N,G}$ we obtain that
$$
\Phi_{M,N,G}\left( \left(\frac{u}{[u]_{M,N,G}} \right)^* \right) = \Phi_{M,N,G}\left( \frac{u^*}{[u]_{M,N,G}} \right) \leq  \Phi_{M,N,G}\left( \frac{u}{[u]_{M,N,G}} \right)  \leq 1.
$$
Again, the definition of the Luxemburg norm yilds
$$
[u^*]_{M,N,G} \leq \inf\left\{ \lam: \Phi_{M,N,G}\left( \frac{u^*}{\lam}\right) \leq 1\right\} \leq [u]_{M,N,G}.
$$
Since $\|u\|_G = \|u^*\|_G$, the result follows.
\end{proof}

\section{A compactness result}
\label{section-compactness}
In order to obtain the compactness in the embedding of $W^{M,N,G}(\R^n)$ into $L^G_{loc}(\R^n)$ we will assume the following additional condition on $M$ and $N$
\begin{equation} \tag{$P_4$}   \label{P4}
\lim_{r\to 0}\frac{N(2r) M(2r)^{p^-}}{r^n} =0.
\end{equation}

The following technical lemma provides the equi-continuity of modulars.
\begin{lema} \label{lema.comp}
Let $0<s<1$ and $G$ be an Orlicz function. Then, 
$$
\Phi_{G} (\tau_h u - u) \leq C \; \frac{N(2|h|) M(2|h|)^{p^-}}{|h|^n} \; \Phi_{M,N,G}(u),
$$
for every $u\in W^{M,N,G}(\R^n)$ and every $0<|h|<\frac12$, where $\tau_h u(x) =u(x+h)$ and $C=C(n,p^+)$.
\end{lema}

\begin{proof}
Condition  \eqref{D2} gives that
$$
G(| u(x+h)-u(x) |) \leq \C \left[ G(| u(x+h)-u(y) |) + G(| u(y)-u(x) |) \right]
$$
for all $y\in B_{|h|}(x)$. Then
\begin{align} \label{ec1.l1}
\begin{split}
\Phi_{G} (\tau_h u - u) &= \frac{1}{|B_{|h|}(x)|}\int_{B_{|h|}(x)}\intr G(| u(x+h)-u(x) |) \,dx\,dy\\
&\leq \frac{\C}{|h|^n\omega_n } \int_{B_{|h|}(x)}\intr G(| u(x+h)-u(y) |) \,dx\,dy\\
&\quad +\frac{\C}{|h|^n\omega_n }  \int_{B_{|h|}(x)}\intr G(| u(y)-u(x) |) \,dx\,dy  \\
&= \frac{\C}{|h|^n \omega_n }( I_1+I_2 ).
\end{split}
\end{align}
Given $x\in \R^n$ and $y\in B_{|h|}(x)$ we have that
$$
|x-y|\leq |h|, \qquad |x+h-y|\leq |x-y|+|h|\leq 2|h|.
$$
Then, the integral $I_1$ can be bounded as  
\begin{align*}
I_1&= \int_{B_{|h|}(x)}\intr G\left(\frac{| u(x+h)-u(y) |}{M(|x+h-y|) }M(|x+h-y|)\right) N(|x+h-y|) \frac{dx\,dy}{N(|x+h-y|)}\\
&\leq N(2|h|) \int_{B_{|h|}(x)}\intr G\left(\frac{|u(x+h)-u(y) |}{M(|x+h-y|) }M(2|h|)\right)  \frac{dx\,dy}{N(|x+h-y|)}\\
&\leq N(2|h|) M(2|h|)^{p^-} \Phi_{M,N,G}(u).
\end{align*}
Analogously, 
$$
I_2\leq  N(2|h|) M(2|h|)^{p^-} \Phi_{M,N,G}(u).
$$
Finally, inserting the two upper bounds found above in \eqref{ec1.l1} we obtain that
$$
\Phi_G (\tau_h u - u) \leq  C  \; \frac{N(2|h|) M(2|h|)^{p^-}}{|h|^n}  \;  \Phi_{M,N,G}(u)
$$
and the lemma follows.
\end{proof}

If we further assume condition \eqref{P4} on $M$ and $N$, by using Lemma \ref{lema.comp} and the same arguments that in \cite[Theorem 3.1]{FBS}, we can apply  a variant of the well-known Fr\`echet-Kolmogorov compactness theorem to obtain the following compactness result.

\begin{thm} \label{teo.comp}
Let $0<s<1$, $G$ an Orlicz function and $M$,$N$ satiafying conditions \eqref{cond.intro}--\eqref{P4}. Then for every $\{u_n\}_{n\in\N}\subset W^{M,N,G}(\R^n)$ a bounded sequence, i.e., $\sup_{n\in\N} ( \Phi_{M,N,G}(u_n) + \Phi_G(u_n) )<\infty$, there exists $u\in W^{M,N,G}(\R^n)$ and a subsequence $\{u_{n_k}\}_{k\in\N}\subset \{u_n\}_{n\in\N}$ such that $u_{n_k}\to u$ in $L^G_{\text{loc}}(\R^n)$.
\end{thm}

\subsection{Examples} 
Condition \eqref{P4} is fulfilled in the examples introduced in Section \ref{subsection-examples}.
\begin{itemize}
\item When $M(r)=r^s$, $0<s<1$ and $N(r)=r^n$, $n\geq 1$ we obtain the fractional Orlicz-Sobolev spaces defined in \cite{FBS}. In this case
$$
\lim_{r \to 0}\frac{N(2r) M(2r)^{p^-}}{r^n}= C \lim_{r \to 0} r^{sp^-} =0.
$$

\item When $M(r)=r$ and $N(r)=r^{n-1}$, $n\geq 1$ we obtain the Orlicz-Slobodetskii spaces defined in \cite{KK}. In this case
$$
\lim_{r \to 0}\frac{N(2r) M(2r)^{p^-}}{r^n}= C \lim_{r \to 0} r^{p^- -1} =0.
$$

\item When $M(r)=r^s (1+|\log r|)^\beta$, $\beta>0$,  and $N(r)=r^n$  we obtain the family of weighted Besov Spaces considered in \cite{A}. Indeed, \eqref{P1} and \eqref{P2} easily hold and \eqref{P3} is fulfilled since 
$$
\lim_{r \to 0}\frac{N(2r) M(2r)^{p^-}}{r^n}= C \lim_{r \to 0} r^{sp^-} |\log(1+r)|^{\beta p^-} =0.
$$

\item When $M(r)=r^sG^{-1}(r^n)$ and $N(r)=1$, $n\geq 1$ we obtain the Orlicz-Slobodetskii  spaces defined in \cite{ABS}.  Indeed, \eqref{P1} and \eqref{P2} easily hold and \eqref{P3} reads as
 
$$
\lim_{r \to 0}\frac{N(2r) M(2r)^{p^-}}{r^n}= C \lim_{r \to 0} r^{sp^- -n} (G^{-1}(t^n))^{p^-} \leq C\lim_{r \to 0} r^{sp^--n}
r^\frac{np^-}{p^+} =0
$$
if and only if 
$$
\frac{n}{s}< \frac{p^-}{p^+-p^-},
$$
where we have used that $G^{-1}(r)\leq \max\{r^{1/p^+}, r^{1/p^-}\}$.
\end{itemize}

\section{Applications to Poincar\'e's constants and nonlinear eigenvalue problems} \label{aplicaciones}
As a corollary of the P\'{o}lya-Szeg\"{o} principle stated in Theorem \ref{main} we obtain a Faber-Krahn type inequality for the Poincar\'e inequality in $W^{M,N,G}_0(\Omega)$ if $G$ is assumed to satisfy the growing condition \eqref{cond}.

\begin{cor} \label{cor1}
Let $G$ be a Young function satisfying \eqref{cond} and $M$, $N$ satisfying \eqref{P1}--\eqref{P3} For every $\Omega\subset \R^n$ open and bounded we have that
$$
\alpha_1^G(B)\leq \alpha_1^G(\Omega)
$$
where $B$ is a ball with $\mathcal{L}^n(B)=\mathcal{L}^n(\Omega)$. 
\end{cor}

Under some extra convexity assumptions on $G$, $M$ and $N$, a Faber-Krahn inequality holds for the principal eigenvalue of $(-\Delta_g)^{M,N}$. Precisely,
\begin{thm}\label{one-reflection-h}
Let $G$ be a Young function and $M, N$ satisfying \eqref{P1}--\eqref{P3}. Assume moreover that $h(t):= tg(t)$ is convex. If $u \in W^{M,N,G}(\R^n)$ then we have that
$$
\langle (-\Delta_g)^{M,N} u_*,u_* \rangle  \leq  \langle (-\Delta_g)^{M,N} u,u \rangle
$$
\end{thm}

\begin{proof}
In view of \eqref{fracg} we can write
\begin{align*}
\langle (-\Delta_g)^{M,N} u_H,u_H \rangle &= 
\int_{\R^{2n}} g\left(|D_M u_H|\right) |D_M u_H|\, d\mu_N\\
&= \int_{\R^{2n}} h\left(|D_M u_H|\right) \, d\mu_N.
\end{align*}
Hence, since $h$ is continuous and convex,  by Proposition \ref{one-reflection} we obtain that
$$
\langle (-\Delta_g)^{M,N} u_H,u_H \rangle  \leq  \langle (-\Delta_g)^{M,N} u,u \rangle.
$$
Then, proceeding as in Theorem \ref{main} the result follows.
\end{proof}

\begin{cor} \label{cor2}
Let $G$ be an Young function with $g=G'$ such that $h(t)=tg(t)$ is convex and $M, N$ satisfying \eqref{P1}--\eqref{P4}. For every $\Omega\subset \R^n$ open and bounded we have that
$$
\lam_1^G(B)\leq \lam_1^G(\Omega)
$$
where $B$ is a ball with $\mathcal{L}^n(B)=\mathcal{L}^n(\Omega)$. 
\end{cor}

\section*{Acknowledgements}

This paper is partially supported by grants UBACyT 20020130100283BA and 20020160100002BA , CONICET PIP 11220150100032CO 
and 11220130100006CO, and by ANPCyT PICT 2012-0153 and 2014-1771.

All of the authors are members of CONICET.

\bibliographystyle{amsplain}
\bibliography{biblio}

\end{document}